\documentclass{amsart}

\usepackage{enumerate}

\usepackage{amssymb}
\usepackage{graphicx}
\usepackage{MnSymbol}

\usepackage{amsthm}

\title{The fundamental group of reduced suspensions}
\author{Samuel M. Corson}
\author{Wolfram Hojka}

\theoremstyle{definition}
\theoremstyle{definition}\newtheorem*{B}{Theorem \ref{thmOmega}}

\theoremstyle{definition}\newtheorem{bigtheorem}{Theorem}

\theoremstyle{definition}\newtheorem{theorem}{Theorem}[section]

\theoremstyle{definition}
\theoremstyle{definition}\newtheorem{proposition}[theorem]{Proposition}
\theoremstyle{definition}\newtheorem{definition}[theorem]{Definition}
\theoremstyle{definition}
\theoremstyle{definition}
\theoremstyle{definition}\newtheorem{remark}[theorem]{Remark}
\theoremstyle{definition}
\theoremstyle{definition}\newtheorem{lemma}[theorem]{Lemma}
\theoremstyle{definition}
\theoremstyle{definition}
\theoremstyle{definition}
\theoremstyle{definition}
\theoremstyle{definition}\newtheorem{construction}[theorem]{Construction}
\newtheorem*{question*}{Question}
\newtheorem*{theorem*}{Theorem}
\newtheorem*{corollary*}{Corollary}
\newtheorem*{lemma*}{Lemma}

\def\pmc#1{\setbox0=\hbox{#1}
    \kern-.1em\copy0\kern-\wd0
    \kern.1em\copy0\kern-\wd0}

\newcommand{\card}{\operatorname{card}}

\newcommand{\K}{\mathcal{K}}

\DeclareMathOperator{\im}{im}

\DeclareMathOperator{\topprod}{\circledast}

\begin{document}

\address{Matematika Saila\\ UPV/EHU, Sarriena S/N\\
48940, Leioa - Bizkaia, Spain}
\email{sammyc973@gmail.com}

\address{Institute for Analysis and Scientific Computation\\
Technische Universit\"at Wien, Vienna, Austria
}

\email{w.hojka@gmail.com}
\thanks{The work of the first author was supported by the Additional Funding Programme for Mathematical
Sciences, delivered by EPSRC (EP/V521917/1) and the Heilbronn Institute
for Mathematical Research.  Also by a Maria Zambrano grant MAZAM22/08 at UPV/EHU funded by the Ministry of Universities and financed by European Union - Next Generation EU.  Also by the Basque Government Grant IT1483-22 and Spanish Government Grants PID2019-107444GA-I00 and PID2020-117281GB-I00}

\maketitle

\begin{abstract}  We classify pointed spaces according to the first fundamental group of their reduced suspension.  A pointed space is either of so-called totally path disconnected type or of horseshoe type.  These two camps are defined topologically but a characterization is given in terms of fundamental groups.  Among totally path disconnected spaces the fundamental group is shown to be a complete invariant for a notion of topological equivalence weaker than that of homeomorphism.

\end{abstract}

\begin{section}{Introduction}

In general, the calculation of homotopy groups of suspensions and reduced suspensions is an amazingly difficult endeavour, and little is known, even for simple spaces like spheres. A great deal of research has been devoted to this, beginning with Freudenthal's celebrated stabilization theorem \cite{F}. For the fundamental group only, on the other hand, the situation is more clearly laid out as long as one only considers CW-complexes, or more generally well-pointed spaces: $\pi_1 \Sigma X$ is just a free group with the number of connected components of $X$ minus one as its rank.

This, however, belies the rich structure and great variety of groups that can be attained if arbitrary spaces are considered. For example, the \emph{Hawaiian earring} with its rather complicated fundamental group is the reduced suspension of $\omega +1$, i.e.~a countable set with one limit point. An observer might be inclined to assume that $\pi_1\Sigma X$ is at least completely determined by the path components $\pi_0 X$, as is true for the ordinary suspension, perhaps additionally equipped with an appropriate topology -- but even that cannot capture the entire variability. As a second example, the reduced suspension of the topologist's sine curve (which has two path components) is known as the \emph{harmonic archipelago} and its fundamental group is uncountable and includes the rationals $\mathbb Q$ as a subgroup (this space is homotopy equivalent to that mentioned in \cite{BS} of the same name).

In this article we offer a criterion that relates algebraic properties of $\pi_1\Sigma X$, such as \emph{not} containing divisible elements, or being isomorphic to some $\pi_1\Sigma Y$ with $Y$ a totally path disconnected space, and a topological condition on $X$ to each other (Theorem \ref{thmOmega}). For the case of totally path disconnected spaces we provide a topological characterization that precisely determines the isomorphism classes of these groups in Theorem \ref{tpd-sbc}.

\begin{bigtheorem}\label{tpd-sbc}
Let $(X,x)$, $(Y,y)$ be totally path disconnected Hausdorff spaces which are first countable at their distinguished points. Then $\pi_1(\Sigma(X,x)) \simeq \pi_1(\Sigma(Y,y))$ if and only if there is a bijection $f: (X, x) \rightarrow (Y, y)$ with $f$ continuous at $x$ and $f^{-1}$ continuous at $y$.
\end{bigtheorem}

\begin{definition}
We will call a pointed space $(X,x)$ a \textbf{horseshoe space} if there exists a neighbourhood $U$ of $x$, and sequences $(x_n)_n$, $(y_n)_n$ in $U$ converging to $x$, so that  $x_n$ and $y_n$ are in the same path-component of $X$ but are in different path components in $U$.
\end{definition}

\begin{bigtheorem}\label{thmOmega}
For a pointed Hausdorff space $(X,x)$ which is first countable at the distinguished point the following properties are equivalent:
\begin{enumerate}
\item
$(X, x)$ is a horseshoe space
\item
The fundamental group $\pi_1(HA)$ of the harmonic archipelago embeds into $\pi_1(\Sigma(X,x))$
\item
The group of rationals $\mathbb Q$ embeds into $\pi_1(\Sigma(X,x))$
\item
$\pi_1(\Sigma(X,x))$ contains an infinitely divisible element

\item
$X$ is not of tpd-type, i.e.~there is no totally path disconnected space $Y$ such that $\pi_1(\Sigma(X,x)) \simeq \pi_1(\Sigma(Y,y))$
\end{enumerate}
\end{bigtheorem}

In Section \ref{Prelim} we give some necessary preliminaries and establish notation.  Theorem \ref{tpd-sbc} is proved in Section \ref{tpd} and Theorem \ref{thmOmega} is proved in Section \ref{Omega}.

We point out that this interesting behavior of reduced suspensions is not limited to the first homotopy group.  We illustrate this using a modification of the Warsaw circle.  Let $X_0 = \{(x, \sin(\pi/x)): x\in (0, 1]\}\subset \mathbb{R}^2$ and $X_1 = \{0\} \times [-1, 1]$, so that $X_0 \cup X_1$ is the topologist's sine curve.  Let $X_2$ be an arc beginning at $(0, -1)$ and ending at $(1, 0)$ such that $X_2 \cap (X_0\cup X_1) = \{(0, -1), (1, 0)\}$ and $X_2$ never enters the upper half plane in $\mathbb{R}^2$, so that the Warsaw circle $WS$ is given by $WS = X_0\cup X_1\cup X_2$.  

Define the \emph{Polish sausage} to be the quotient space $PS = W \times S^1/ \{(0, y, t)\sim (0, y, t')\}_{(0, y)\in X_1; t, t'\in S^1}$  of $W \times S^1$ which identifies each circle above the limit arc $X_1$ to a point without identifying points on circles above different points of the limit arc.  Thus $PS$ is essentially a two dimensional tube which approaches and thins towards a limiting arc from the side and also tapers to the limiting arc from below \'a la the Warsaw circle (see Figure \ref{polska}).  This space is path connected and all points except those on the limit arc have a neighborhood which is homeomorphic to $\mathbb{R}^2$.

\begin{figure}\label{polska}
\centering
\includegraphics[width=5cm]{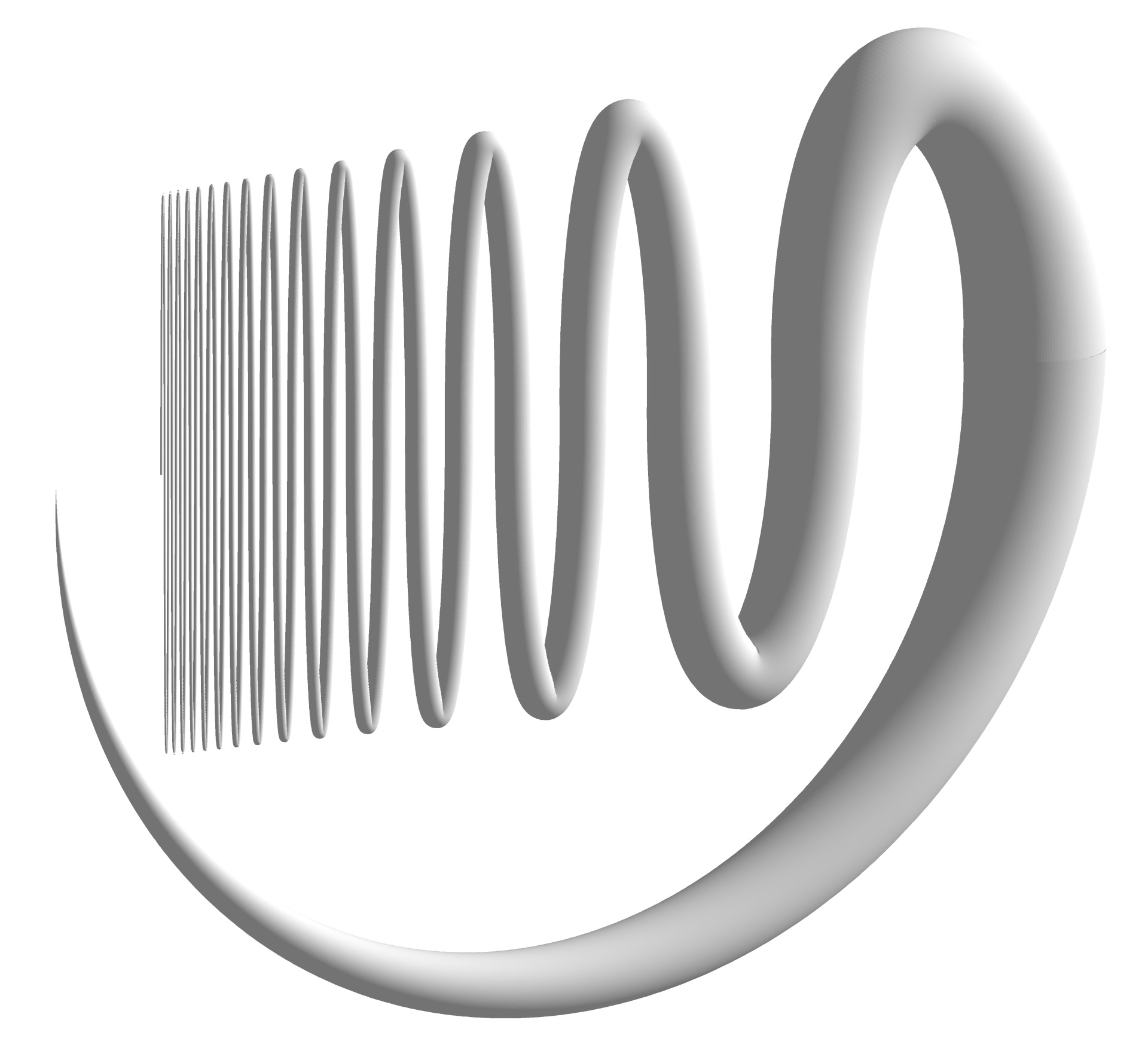}
\caption{Polish sausage}

\end{figure}

We claim, without formal proof, that the Polish sausage has trivial homotopy groups for every dimension $n\geq 1$.  One can roughly see this by noticing that for any compact path connected, locally path connected subset $Y\subset PS$ there exists a homotopy $H: PS \times [0, 1] \rightarrow PS$ such that $H(z, 0) = z$ for all $z\in PS$ and $H(Y, 1) \subseteq X_1$.

However, the second homotopy group of the reduced suspension $\pi_2\Sigma PS$ will be nontrivial.   We sketch why this is the case.  Let $z$ denote the point $[(0, 0)\times S^1]$ in $PS$. The set of disjoint circles $\{(\frac{1}{n+1}, 0)\times S^1\}_{n\in \omega}$ limits to $z$.   Taking one of these circles $C_n = (\frac{1}{n+1}, 0)\times S^1$ together with the point $z$ and computing the reduced suspension $\Sigma(C_n \cup \{z\}, z)$ we get the quotient space of a $2$-dimensional sphere obtained by identifying the north and south poles.  This reduced suspension includes into the reduced suspension $\iota_n:\Sigma(C_n \cup \{z\}, z) \rightarrow \Sigma(PS, z)$ of the entire space.  For each $n\in \omega$ let $\rho_n: [0, 1]^2\rightarrow \Sigma(C_n\cup \{z\},z)$ be a representative of any nontrivial element of $\pi_2(\Sigma(C_n\cup \{z\}, z))$.  Each $f_n = \iota_n\circ \rho_n$ is nulhomotopic in the space $\Sigma(PS, z)$, but the infinite concatenation  $f_0\ast f_1\ast f_2\ast \cdots :[0, 1]^2 \rightarrow \Sigma(PS, z)$ will not be.

\end{section}

\begin{section}{Notation and Preliminaries}\label{Prelim}

Let us first establish some notation. Set $I= [0,1]$, $D_1 = [-1, 1]$, $\partial D_1= \{-1, 1\}$, and $D_1^{\circ} = (-1, 1)$.  The reduced suspension $\Sigma(X,x)$ of a pointed space is the quotient space of $X \times D_1$ where $(X \times \partial D_1) \cup (\{x\} \times D_1)$ is collapsed to a single point. By abusing language, we may identify this point with $x$. Further $\mathring{X} := (X\setminus \{x\}) \times D_1^\circ$ canonically embeds into $\Sigma(X,x)$, as does $\check X := X \times \{0\}$, a homeomorphic copy of $X$.  We shall occasionally refer to subsets of $\Sigma(X, x)$ of the form $X\times ([-1, a) \cup (a, 1])$ and consider that $x$ belongs to this set.  Let $\chi: \Sigma(X,x) \rightarrow X$ be the noncontinuous map that takes $x$ to itself and a point in $\mathring{X}$ to its projection in $X$.  Let $\check{\iota}: X \rightarrow \check{X}$ be the map $\check{\iota}(z) = (z, 0) \in \check{X}$.  Thus $\check{\iota}$ is a homeomorphism and $\chi\upharpoonright\check{X}$ is the inverse homeomorphism.

\begin{definition}
A map $g: Z \rightarrow \Sigma(X,x)$ is \textbf{reduced} if the image under $g$ of each connected component $K$ of $Z\setminus g^{-1}(x)$ has nonempty intersection with $\check X$.
\end{definition}

\begin{lemma}\label{loop-red}
Every based loop in $\Sigma(X,x)$ is homotopic to a reduced loop.
\end{lemma}

\begin{proof}
Let $\gamma:I \rightarrow \Sigma(X, x)$ be a loop at $x$.  Let $\Lambda$ be the set of open intervals $(a,b)$ in $I \setminus \gamma^{-1}(x)$ such that $\gamma((a,b)) \cap \check{X}= \emptyset$.  Notice that for each $\lambda \in \Lambda$ we have that $\gamma(\lambda)$ is either entirely contained in the open upper half of $\mathring{X}$ or the open bottom half.  Let $\Lambda^+$ be the set of all $\lambda \in \Lambda$ with each element of $\gamma(\lambda)$ having strictly positive second coordinate and define $\Lambda^-$ similarly.  Define a homotopy $H: I^2 \rightarrow \Sigma(X, x)$  by letting

\[
H(t,s) = \left\{
\begin{array}{ll}
 \gamma(t) 
                                            & \text{if } t\notin \bigcup\Lambda, \\
(x', (1-s)l + s) 
                                            & \text{if }t\in \lambda\in \Lambda^+ \text{ and }\gamma(t) = (x', l), \\
(x', (1-s)l - s) 
                                            & \text{if } t\in \lambda \in \Lambda^- \text{ and }\gamma(t) = (x', l).
\end{array}
\right.
\]

The check that $H$ is continuous and that $H(t, 1)$ is reduced is straightforward.
\end{proof}

\begin{lemma}\label{homot-red}
If two reduced loops in $\Sigma(X,x)$ are homotopic the homotopy can be chosen to be reduced.
\end{lemma}
\begin{proof}
Let $\gamma_0$ and $\gamma_1$ be reduced loops and $H$ be a homotopy from $\gamma_0$ to $\gamma_1$ in $\Sigma(X, x)$.  Let $H'(t, s) = H(t,s)$ whenever $(t,s)$ is in a component $K$ of $I^2 \setminus H^{-1}(x)$ which intersects $\check{X}$ and let $H'(t,s)= x$ otherwise.  This new map $H'$ is continuous and reduced.
\end{proof}

We now define and discuss the topologist product operator on groups, which was first defined in \cite{E1}.  Although this construction is fairly general it is sufficient for  our purposes to only consider the topologist product of countably many groups.  If $(G_n)_n$ is a sequence of groups we define the topologist product $\topprod_n G_n$ in the following manner.  Call a map $W:\overline{W} \rightarrow \sqcup_{n\in \mathbb{N}}G_n$ a \textbf{word} if $\overline{W}$ is a countable totally ordered set and for each $n \in \mathbb{N}$ the set $\{i \in \overline{W}: W(i) \in G_n\}$ is finite.  We consider two words $W_0, W_1$ to be equal provided there exists an order isomorphism $f:\overline{W_0} \rightarrow \overline{W_1}$ such that $W_1(f(i)) = W_0(i)$.  Let $\mathcal{W}(G_n)_n$ denote the collection of words.  For each $N\in \mathbb{N}$ define a map $p_N$ via the restriction $p_N(W) = W\upharpoonright\{i\in \overline{W}: (\exists n \leq N) W(i) \in G_n\}$.  Thus $p_N$ maps elements of $\mathcal{W}(G_n)_n$ to finite words which only utilize elements of the set $\bigcup_{n\leq N}G_n$.  Define an equivalence $\sim$ on words by letting $W_0 \sim W_1$ provided for all $N\in \mathbb{N}$ we have $p_N(W_0)$ equals $p_N(W_1)$ in the free product $\ast_{n\leq N}G_n$.  The set $\topprod_n G_n= \mathcal{W}(G_n)_n/ \sim$ has a group structure.  The binary operation is given by word concatenation: $[W_0][W_1] = [W_0W_1]$ where $W_0W_1$ is the word with domain $\overline{W_0} \sqcup \overline{W_1}$ with order which extends the orders of $\overline{W_0}$ and $\overline{W_1}$ and places elements of $\overline{W_0}$ in front of those in $\overline{W_1}$ and

\[
W_0W_1(i) = \left\{
\begin{array}{ll}
 W_0(i) 
                                            & \text{if } i\in \overline{W_0}, \\
W_1(i)
                                            & \text{if } i\in\overline{W_1}.
\end{array}
\right.
\]

\noindent The trivial element is the equivalence class of the empty word and the inverse of an element is given by $[W]^{-1} = [W^{-1}]$, where $W^{-1}$ is the word having the same domain as $W$ under the reversed order and $W^{-1}(i) = (W(i))^{-1}$.

The word map $p_N$ induces a group homomorphism denoted $p_N$ by abuse of notation which is a retraction of $\topprod_n G_n$ to the free product subgroup $\ast_{n\leq N} G_n$.  If the groups $G_n$ are trivial for all $n >N$ then $p_N$ is an isomorphism of the group $\topprod_n G_n$ with itself.  Given $N\in \mathbb{N}$ let $\topprod_{n > N} G_n$ denote the subgroup of $\topprod_n G_n$  consisting of those elements having a word representative $W \in [W]$ such that $\im(W) \subseteq \sqcup_{n = N+1}^{\infty} G_n$.  There is a canonical isomorphism $\topprod_n G_n \simeq (\ast_{n\leq N} G_n) \ast (\topprod_{n > N}G_n)$ obtained by considering a word $W \in \mathcal{W}(G_n)_n$ as a finite concatenation of equivalence classes in $\ast_{n\leq N} G_n$ and equivalence classes in $\topprod_{n > N}G_n$.  Let $p^N$ denote the retraction from $\topprod_{n}G_n$ to $\topprod_{n > N}G_n$.  In all examples that we consider in this paper the groups $G_n$ will be free groups.

The canonical example of a topologist product has $G_n \simeq \mathbb{Z}$, in which case $\topprod_n G_n \simeq \topprod_n \mathbb{Z}$.  Letting $a_n$ be a choice of generator for $G_n \simeq \mathbb{Z}$, any word $W\in \mathcal{W}(\mathbb{Z})_n$ is $\sim$ to a word $W_0$ where $\im(W_0) \subseteq \{a_n^{\pm 1}\}_{n\in \mathbb{N}}$.   This particular topologist product is isomorphic to the fundamental group of the Hawaiian earring mentioned in the introduction.  We give a description of a class of homomorphisms called realizations from $\topprod_n\mathbb{Z}$ to the fundamental group of a space.  Suppose $(Z, z)$ is a Hausdorff space which is first countable at $z$ and that $(\gamma_n)_n$ is a sequence of loops such that for any neighborhood $U$ of $z$ there is an $N\in \mathbb{N}$ such that $\im(\gamma_n) \subseteq U$ for all $n\geq N$.   Let $\mathcal{I}$ denote the set of closed intervals in $I$ whose endpoints are elements of the Cantor ternary set and whose interiors contain no element of the ternary set.  The set $\mathcal{I}$ has a natural order given by letting $I_0 < I_1$ if and only if all points in the interior of $I_0$ are less than all points in $I_1$, and the order type is isomorphic to that of $\mathbb{Q}$.  Thus given a word $W\in \mathcal{W}(\mathbb{Z})_n$ with $\im(W) \subseteq \{a_n^{\pm 1}\}_{n\in \mathbb{N}}$ there exists an order embedding $\iota: \overline{W} \rightarrow \mathcal{I}$ and we let $\Re_{\iota}(W):I \rightarrow Z$ be given by

\[
\Re_{\iota}(W)(t) = \left\{
\begin{array}{ll}
z 
                                            & \text{if }t\in I \setminus \bigcup\im(\iota), \\
\gamma_n(\frac{t - \min(I_0)}{\max(I_0) - \min(I_0)})
                                            & \text{if }t\in I_0 = \iota(i) \text{ and }W(i) = a_n, \\
\gamma_n^{-1}(\frac{t - \min(I_0)}{\max(I_0) - \min(I_0)})
				& \text{if }t\in I_0 = \iota(i) \text{ and }W(i) = a_n^{-1}.

\end{array}
\right.
\]

The check that the homotopy class of $\Re_{\iota}(W)$ does not depend on the order embedding $\iota$ is straightforward, so that one can write $\Re(W)$ as a loop well defined up to a homotopy taking place entirely in the range of $\Re_{\iota}(W)$ for any embedding $\iota$.  The check that $\Re(W_0)$ is homotopic to $\Re(W_1)$ via a homotopy in $\im(\Re(W_0))\cup \im(\Re(W_1))$ whenever $W_0 \sim W_1$ is more challenging.  Finally one checks that $\Re:\topprod_n \mathbb{Z} \rightarrow \pi_1(Z, z)$ is a homomorphism.

For example, the Hawaiian earring is formally the subspace $$E=\bigcup_{n\in \mathbb{N}} C((0, \frac{1}{n+2}),\frac{1}{n+2})$$ of $\mathbb{R}^2$ where $C(p, r)$ denotes the circle centered at $p$ of radius $r$.  Thus $E$ is a shrinking wedge of countably many circles in the plane with wedge point $(0,0)$.  For each $n\in \mathbb{N}$ we let $\gamma_n:I \rightarrow E$ be the loop $\gamma_n(t) = (\frac{1}{n+2}\sin(2\pi t),\frac{1}{n+2} - \frac{1}{n+2}\cos(2\pi t))$.  Thus the homotopy class $[\gamma_n]$ generates the fundamental group $\pi_1(C((0, \frac{1}{n+2}), \frac{1}{n+2}))$.  The associated realization map is an isomorphism.  The retraction map $p_N: \topprod_n\mathbb{Z} \rightarrow \ast_{n\leq N} \mathbb{Z}$ is the induced map of the topological retraction which fixes the outer circles $\bigcup_{n\leq N} C((0, \frac{1}{n+2}),\frac{1}{n+2})$ and sends all other points to $(0, 0)$.

We note that if $(r_n)_n$ is a sequence of finite natural numbers which is not eventually $0$ then $\topprod_n \mathbb{Z} \simeq \topprod_n F(r_n)$ by enumerating the disjoint union $\sqcup_{n=0}^{\infty} r_n$ via a bijection $g: \mathbb{N} \rightarrow \sqcup_{n=0}^{\infty} r_n$ and defining an isomorphism by mapping $a_m$ to the free generator $g(m)$.

\end{section}

\begin{section}{Proof of Theorem \ref{tpd-sbc}}\label{tpd}
We give the proof of Theorem \ref{tpd-sbc}, beginning with the reverse direction.  Let $f: (X,x) \rightarrow (Y,y)$ be a bijective map between totally path disconnected spaces, with $f$ continuous at $x$ and $f^{-1}$ continuous at $y$. Then there is an induced map $\varphi: \Sigma(X,x) \rightarrow \Sigma(Y,y)$ defined by $\varphi((p,t)):=(f(p),t)$ if $(p,t)\in \mathring{X}$, and $\varphi(x):=y$. This map is well-defined, however it need not be continuous, for neither is $f$. Nevertheless, it will suffice in a more specific setting:

Recall that a \textbf{Peano continuum} is a path connected, locally path connected compact metric space.  The classical Hahn-Mazurkiewicz theorem states that a Hausdorff space is a Peano continuum if and only if it is the continuous image of the closed interval $I$.  Peano continua are obviously closed under continuous Hausdorff images.

\begin{lemma}\label{PeanoReducedCont}
If a continuous map $g$ from a Peano continuum $Z$ to $\Sigma(X,x)$ is reduced, the composite map $\varphi\circ g: Z \rightarrow \Sigma(Y,y)$ is also continuous.
\end{lemma}

\begin{proof}  As $Z$ is first countable, it suffices to show that whenever a sequence $(z_n)_n$ converges to $z\in Z$, the image sequence $(\varphi\circ g(z_n))_n$ converges to $\varphi\circ g(z)$. We study two possibilities:

First assume $T := \{\chi\circ g(z_n) : n \in \mathbb{N}\} \cup \{x\}$ is finite. Then the sequence $(g(z_n))_n$ and its limit $g(z)$ are contained in $\Sigma(T,x)\subseteq \Sigma(X,x)$, a bouquet of finitely many circles.  Clearly, $\varphi$ induces a homeomorphism from $\Sigma(T,x)$ to $\Sigma(f(T),y)$, so the sequence $\varphi \circ g(z_n)$ converges to $\varphi \circ g(z)$.

Otherwise consider an infinite subsequence $(z_{n_k})_k$ of $(z_n)_n$ so that $\chi\circ g(z_{n_k})$ are all distinct and not equal to $x$.  Because $g$ is reduced, we may choose to each $z_{n_k}$ an $m_k$ in the same component of $Z\setminus g^{-1}(x)$ with $g(m_k) \in \check{X}$, or more precisely, $g(m_k) = (\chi \circ g(z_{n_k}), 0)$.  As $Z$ is compact, there is a subsequence $(z_{n_{k_i}})_i$ such that the corresponding $m_{k_i}$ converge. $Z$ is even a Peano continuum, hence there is a path $\rho$ with $\rho \big(1-\frac{1}{i+1}\big) = m_{k_i}$.  Since these points are in different components in $Z$, the path $g\circ \rho$ in $\Sigma(X,x)$ hits the base point $x$ on each subinterval $\big(1-\frac{1}{i+1},1-\frac{1}{i+2}\big)$, so $g\circ \rho(1)$ has to be $x$. This then implies the sequence $(g(m_{k_i}))_i$ converges to $x$, which in turn requires $\chi \circ g(z_{n_{k_i}}) = \chi \circ g(m_{k_i})$ to converge to $x$. In particular this shows that $g(z) = x$. Further, $f \circ \chi \circ g (z_{n_{k_i}}) = \chi \circ \varphi \circ g (z_{n_{k_i}})$ converges to $f(x) = y$, and so also does $\varphi \circ g (z_{n_{k_i}})$.

Combining this with the argument from the first case shows that every subsequence $(z_{n_k})_k$ of $(z_n)_n$ has a sub-subsequence $(z_{n_{k_i}})_i$ with $\varphi \circ g (z_{n_{k_i}}) \rightarrow y$. Fortunately this property is adequate evidence for $(\varphi \circ g (z_n))_n$ converging to $\varphi \circ g(z) = y$.
\end{proof}

We return to the proof of the theorem. By virtue of $\varphi$, every reduced loop $\gamma$ in $\Sigma(X,x)$ is matched up with a reduced loop $\varphi\circ \gamma$ in $\Sigma(Y,y)$. Now by Lemma \ref{loop-red} any element in $\pi_1(\Sigma(X,x))$ can be represented by a reduced loop. Supposing two reduced loops $\gamma, \delta$ represent the same class, then by Lemma \ref{homot-red} there is a reduced homotopy $H$ between them. Clearly, $\varphi\circ H$ is a homotopy between $\varphi\circ \gamma$ and $\varphi\circ \delta$, continuous again thanks to the previous lemma. 

Hence, the map $\psi: \pi_1(\Sigma(X,x)) \rightarrow \pi_1(\Sigma(Y,y))$, given by $[\gamma] \mapsto [\varphi\circ\gamma]$ for reduced representatives $\gamma$, is well-defined. Finally, $\psi([\gamma] [\delta]) = \psi([\gamma]) \psi([\delta])$, evidently by construction, thus $\psi$ is a homomorphism. Since it is invertible, it actually is an isomorphism, $\pi_1(\Sigma(X,x)) \simeq \pi_1(\Sigma(Y,y))$.  This completes the proof of the reverse direction of Theorem \ref{tpd-sbc}.

For the forward direction of Theorem \ref{tpd-sbc} we build up some notation and develop some machinery about sequences of cardinal numbers.  Then we characterize the fundamental group of $(X, x)$, a totally path disconnected Hausdorff space which is first countable at the distinguished point.

If $(\lambda_n)_n$ and $(\mu_n)_n$ are two sequences of cardinal numbers we write $(\lambda_n)_n \approx (\mu_n)_n$ provided there exists a bijection $f:\sqcup_n \lambda_n \rightarrow \sqcup_n \mu_n$ such that for each $M\in \mathbb{N}$ there exists $M'\in \mathbb{N}$ such that $f(\sqcup_{n=0}^M \lambda_n)\subseteq \sqcup_{n=0}^{M'} \mu_n$ and $f^{-1}(\sqcup_{n=0}^{M}\mu_n) \subseteq \sqcup_{n=0}^{M'}\lambda_n$.  The relation $\approx$ is an equivalence relation.  We remark that if $g: \mathbb{N} \rightarrow \mathbb{N}$ is a bijection then given any sequence of cardinals $(\lambda_n)_n$ we have $(\lambda_n)_n \approx (\lambda_{g(n)})_n$.  More generally if $g:\mathbb{N} \rightarrow \mathbb{N}$ is a surjective, finite-to-one function we have $(\lambda_n)_n \approx (\mu_n)_n$ where $\mu_n = \sum_{m\in g^{-1}(n)}\lambda_m$.

\begin{lemma}\label{equivalence}  $(\lambda_n)_n \approx (\mu_n)_n$ if and only if the following hold:

\begin{enumerate}

\item  $(\lambda_n)_n$ is eventually zero iff $(\mu_n)_n$ is

\item  If $\kappa$ is an infinite cardinal then $(\lambda_n)_n$ is eventually $<\kappa$ iff $(\mu_n)_n$ is

\item  For each $M \in \mathbb{N}$ there exists $M'$ such that $\sum_{n=0}^M \lambda_n \leq \sum_{n=0}^{M'}\mu_n$ and $\sum_{n=0}^M \mu_n \leq \sum_{n=0}^{M'}\lambda_n$

\end{enumerate}

\end{lemma}

\begin{proof} ($\Rightarrow$)  Suppose that $(\lambda_n)_n \approx (\mu_n)_n$ and select $f$ as in the definition of $\approx$.  If without loss of generality $\lambda_n = 0$ for all $n \geq M$ then selecting $M'$ as in the definition of $\approx$ we see that $f(\sqcup_{n=0}^M \lambda_n) \subseteq \sqcup_{n=0}^{M'}\mu_n$ and since $f$ is onto we have that $\mu_n = 0$ for $n > M'$.  Thus (1) holds.  If without loss of generality we have $\lambda_n <\kappa$ for $n \geq M$ then we again select $M'$ as in the definition of $\approx$.  If $\mu_m \geq \kappa$ with $m>M'$ then since $f^{-1}$ is bijective we must have $f^{-1}(\mu_m) \subseteq \sqcup_{n=M+1}^{\infty} \lambda_n$.  Since $\lambda_n<\kappa$ for all $n \geq M+1$ we see that $f^{-1}(\mu_m)$ is not a subset of $\sqcup_{n=M+1}^L \lambda_n$ for any $L\in \mathbb{N}$, so $f^{-1}(\mu_m)$ is not a subset of $\sqcup_{n=0}^L \lambda_n$ for any $L\in \mathbb{N}$, a contradiction.  Thus (2) holds.  That (3) holds is clear from the existence of $f$.

($\Leftarrow$)  Suppose the sequences $(\lambda_n)_n$ and $(\mu_n)_n$ satisfy (1)-(3).  We treat cases.  If one, and therefore both, is eventually zero then we can select $M$ such that $\lambda_n = 0 = \mu_n$ for $n\geq M$.  By (3) we know that $\sum_{n\leq M}\lambda_n = \sum_{n\leq M}\mu_n$ and letting $f$ be any bijection $f:\sqcup_{n=0}^M \lambda_n \rightarrow \sqcup_{n=0}^M \mu_n$ it is clear that $f$ witnesses that $(\lambda_n)_n \approx (\mu_n)_n$.

Next, suppose that both sequences are eventually finite and not eventually zero.  If the $\lambda_n$ are all finite then the $\mu_n$ must all be finite as well, by (3).  In this case we have $\sum_{n}\lambda_n = \aleph_0 = \sum_{n}\mu_n$ and any bijection $f: \sqcup_{n=0}^{\infty}\lambda_n \rightarrow \sqcup_{n=0}^{\infty}\mu_n$ demonstrates $(\lambda_n)_n \approx (\mu_n)_n$.  The case where all the $\mu_n$ are finite is symmetric.  In case the $\lambda_n$ (and therefore also the $\mu_n$) are not all finite, select $M$ such that $n\geq M$ implies $\lambda_n$ and $\mu_n$ are both finite.  By assumption we have $\sum_{n=0}^M\lambda_n$ and $\sum_{n=0}^M\mu_n$ are both infinite cardinals, and (3) implies that $\sum_{n=0}^M\lambda_n=\sum_{n=0}^M\mu_n$.  In this case we start with $f':\sqcup_{n=0}^M \lambda_n \rightarrow \sqcup_{n=0}^M \mu_n$ being any bijection and as $\sum_{n= M+1}^{\infty} \lambda_n = \aleph_0 = \sum_{n= M+1}^{\infty} \mu_n$ we can extend such an $f'$ to a bijection $f: \sqcup_{n=0}^{\infty}\lambda_n \rightarrow \sqcup_{n=0}^{\infty}\mu_n$, and this $f$ witnesses $(\lambda_n)_n \approx (\mu_n)_n$.

In our last case we suppose that both sequences are not eventually finite.  We first modify the sequences to make the proof simpler.  Let $n_0<n_1<n_2<\cdots$ be the subsequence in $\mathbb{N}$ such that $\lambda_{n_k}$ is infinite for each $k\in \mathbb{N}$.  We notice that the sequence $(\lambda_k')_k \approx (\lambda_n)_n$ where $\lambda_0' = \sum_{n=0}^{n_0}\lambda_n$ and for $k\geq 1$ we have $\lambda_{k}' = \sum_{n = n_{n_{k-1}}}^{n_{k}}\lambda_k$.  Making the same modification to the $\mu_n$, the conditions (1)-(3) are preserved for the pair of sequences $(\lambda_n')_n$ and $(\mu_n')_n$.  Thus we may assume that the $\lambda_n$ and $\mu_n$ are always infinite.  Let $\kappa$ be the least infinite cardinal such that $(\lambda_n)_n$ is eventually strictly less than $\kappa$, so that by (2) we know that $\kappa$ satisfies the same condition for $(\mu_n)_n$.  Note also that we may assume that $\kappa>\lambda_n, \mu_n$ for $n\geq 1$.  This is because we may select $M$ large enough that $\lambda_n, \mu_n<\kappa$ for all $n>M$ and letting $\lambda_0' = \sum_{n=0}^M \lambda_n$ and $\lambda_k' = \lambda_{k + M}$ for $k>0$.  Defining $\mu_n'$ similarly we get new sequenes $(\lambda_k')_k$ and $(\mu_k')_k$ such that $(\lambda_n)_n \approx (\lambda_k')_k$, $(\mu_k')_k \approx (\mu_n)_n$ and the sequence pair $(\lambda_k')_k$ and $(\mu_k')_k$ still satisfies (1)-(3).  We may also assume that $\lambda_1 \leq \lambda_2\leq \lambda_3<\cdots$ and similarly $\mu_1\leq \mu_2 \leq \mu_3 \leq \cdots$.  This we can do by noticing that since $\kappa >\lambda_n$ for all $n \geq 1$ we may select a subsequence starting $n_0 = 0, n_1 = 1$ and letting $n_2$ be the least index $>1$ such that $\lambda_{n_2}\geq \lambda_{n_1}$.  In general for $k \geq 1$ we let $n_{k+1}$ be the least index such that $\lambda_{n_{k+1}} \geq \lambda_{n_k}$.  Such selections are always possible since $\kappa$ was chosen to be the smallest cardinal which was eventually strictly larger than the $(\lambda_n)_n$ and the $(\mu_n)_n$.  Let $\lambda_0' = \lambda_0$, $\lambda_1' = \lambda_1$ and $\lambda_k' = \sum_{n = n_{k-1}}^{n_k}\lambda_n$ for $k \geq 2$.  Performing the same modifications on the sequence $(\mu_n)_n$ gives us new sequences $(\lambda_k')_k \approx (\lambda_n)_n$, $(\mu_k')_k \approx (\mu_n)_n$ with $(\lambda_k')_k$ and $(\mu_k')_k$ satisfying (1)-(3) with respect to each other.  Thus we may assume that $(\lambda_n)_n \approx (\mu_n)_n$, that all $\lambda_n$ and $\mu_n$ are infinite, that $\kappa$ is the least cardinal greater than $\lambda_n$ and than $\mu_n$ for $n\geq 1$, and that $\lambda_1 \leq\lambda_2\leq \cdots$ and $\mu_1\leq \mu_2 \leq \cdots$.  It is possible that $\lambda_0> \lambda_1$ and that $\mu_0 > \mu_1$.  We treat subcases.

\textbf{Subcase 1}  If $(\lambda_n)_n$ is eventually constant then $\kappa$ is a successor cardinal, say $\kappa = \eta^+$, and this implies that $(\mu_n)_n$ also stabilizes at $\eta$.  If in addition we have $\lambda_0\geq \kappa$ then by (3) we know that $\mu_0 = \lambda_0 \geq \kappa$.  Letting $(\lambda_k')_k$ be given by $\lambda_0' = \lambda_0$ and $\lambda_k' = \eta$ for $k \geq 1$ it is easy to see that $(\lambda_n)_n \approx (\lambda_k')_k \approx (\mu_n)_n$ and we have $(\lambda_n)_n \approx (\mu_n)_n$.  If instead $\lambda_0 <\kappa$ then by (3) we know that $\mu_0< \kappa$ as well.  Letting $\lambda_k' = \eta$ for all $k\in \mathbb{N}$ we get once more that $(\lambda_n)_n \approx (\lambda_k')_k \approx (\mu_n)_n$ and so we are done with this subcase.

\textbf{Subcase 2}  If $(\lambda_n)_n$ is not eventually constant then $\kappa $ is a limit cardinal.  Then $(\mu_n)_n$ cannot stabilize either and we have $\kappa = \sup\{\lambda_n\}_{n=1}^{\infty} = \sup\{\mu_n\}_{n=1}^{\infty}$.  If in addition we have $\lambda_0 \geq \kappa$ then by (3) we have $\mu_0 = \lambda_0 \geq \kappa$.  By gathering finite segments as before we may assume that $\lambda_1< \mu_1<\lambda_2 < \mu_2< \cdots$.  Let $f:\sqcup_{n=0}^{\infty}\lambda_n \rightarrow \sqcup_{n=0}^{\infty}\mu_n$ be defined by letting $f\upharpoonright\lambda_0$ be a bijection of $\lambda_0$ with $\mu_0$, $f(\lambda_1)\subset\mu_1$, $f^{-1}(\mu_1 \setminus f(\lambda_1)) \subset \lambda_2$, $f(\lambda_2 \setminus f^{-1}(\mu_1)) \subset \mu_2$, etc.  (such an $f$ is defined by induction).  This gives $(\lambda_n)_n \approx (\mu_n)_n$.  Finally, if $\lambda_0<\kappa$ then $\mu_0<\kappa$ as well (by (3)) and we perform a similar back-and-forth argument to obtain $(\lambda_n)_n \approx (\mu_n)_n$.

\end{proof}

\begin{lemma}\label{onlycompact} If $\gamma: I \rightarrow \Sigma(X, x)$ is a based loop then the only limit point of $\gamma(I) \cap \check X$ in $\Sigma(X, x)$ can be $x$.
\end{lemma}

\begin{proof}  The space $\Sigma(X, x)$ is Hausdorff since $X$ is Hausdorff.  Then $\gamma(I)$ is a Peano continuum, and in particular is locally path connected and compact.  The set $\check X$ is closed in $\Sigma(X, x)$ and so $\check X \cap \gamma(I)$ is compact.  Thus $\check X \cap \gamma(I)$ is closed in $\Sigma(X, x)$.  If $x'$ is a limit point of $\check X \cap \gamma(I)$ in $\Sigma(X, x)$ then $x'\in \check X \cap \gamma(I)$, and if $x' \neq x$ then $x'\in \mathring{X}$.  As $\gamma(I)$ is locally path connected there exists a path connected neighborhood of $x'$ inside $\mathring{X} \cap \gamma(I)$.  This is impossible since any neighborhood in $\mathring{X}$ of $x'$ contains a point in $\gamma(I) \cap \check X \setminus \{x'\}$ and such a point cannot be connected to $x'$ by a path in $\mathring{X}$ since $X$ is totally path disconnected.
\end{proof}

If $x\in K\subseteq X$ we may consider $\Sigma(K, x)$ to be a subspace of $\Sigma(X, x)$.  There is a map $p_{K}: \pi_1(\Sigma(X,x)) \rightarrow \pi_1(\Sigma(K, x))$ given by $[\gamma] \mapsto [\gamma']$ where

\[
\gamma'(t) = \left\{
\begin{array}{ll}
\gamma(t)
                                            & \text{if }\gamma(t) \in \Sigma(K, x), \\
x
                                            & \text{otherwise}.
\end{array}
\right.
\]

\noindent It is straightforward to prove that this map is a well-defined homomorphism.  Moreover, this map composes with the homomorphism $i_{*}: \pi_1(\Sigma(K, x)) \rightarrow \pi_1(\Sigma(X, x))$ induced by inclusion to give the identity map on $\pi_1(\Sigma(K, x))$.  Thus $p_K$ is a group theoretic retraction, so that we may write $\pi_1(\Sigma(K, x)) \leq \pi_1(\Sigma(X, x))$.  Also, if $x\in K_1\subseteq K_2 \subseteq X$ we have $p_{K_1}\circ p_{K_2} = p_{K_1}$.  

Let $\K_X$ be the set of all countable compact subsets of $X$ which contain $x$ and whose set of limit points is included in $\{x\}$.  Alternatively $K\in \K_X$ if and only if $x\in K$ and for any neighborhood $U$ of $x$ the set $K\setminus U$ is finite.  The set $\K_X$ is a directed set under set inclusion and upper bounds of finite subsets of $\K_X$ are given by unions.

\begin{lemma}\label{unionofsubgroups}  $\pi_1(\Sigma(X, x)) = \bigcup_{K\in \K_X} \pi_1(\Sigma(K, x))$.
\end{lemma}

\begin{proof}  That the expression even makes sense and that the inclusion $\supseteq$ holds follows from the discussion above.  Any based loop $\gamma$ in $\Sigma(X, x)$ is homotopic to a reduced $\gamma_0$ by Lemma \ref{loop-red}.  The image $\gamma_0(I)$ is a Peano continuum and $\gamma_0(I) \cap \check X$ is compact metrizable.  Given a neighborhood $U \subseteq \check X$ of $x$ the set $(\gamma_0(I) \cap \check X ) \setminus U$ must be finite, for otherwise we could find (by compact metrizability) a limit point in $(\gamma_0(I) \cap \check X ) \setminus U$, contrary to Lemma \ref{onlycompact}.  Since $X$ is first countable at $x$ we therefore see that $\gamma_0(I) \cap \check X$ is countable (and again, compact).  Then $$[\gamma] = [\gamma_0] \in \pi_1(\Sigma(\chi(\gamma_0(I) \cap \check{X}), x)) \subseteq \bigcup_{K\in \K_X} \pi_1(\Sigma(K, x)).$$
\end{proof}

\begin{lemma}  If $x$ is not a limit point of $X$ then $\pi_1(\Sigma(X, x))$ is a free group of rank $\card(X) -1$.
\end{lemma}

\begin{proof}  There exists an obvious homomorphism $\phi: F(X \setminus \{x\}) \rightarrow \pi_1(\Sigma(X, x))$ which takes each generator $x'\in X\setminus \{x\}$ to a loop that traverses a generator of $\pi_1(\Sigma(\{x, x'\}, x))$.  By Lemma \ref{unionofsubgroups} this map is necessarily onto, since each $K \in \K_X$ is finite and $\Sigma(K, x)$ is a finite wedge of $\card(K) -1$ many circles whose fundamental group is mapped onto isomorphically from the free subgroup $F(K\setminus \{x\}) \leq F(X \setminus \{x\})$ via the restriction of $\phi$.  In fact the kernel of $\phi$ is trivial since the restriction of $\phi$ to any subgroup $F(K \setminus \{x\})$ is an injection for $K\in \K_X$.  
\end{proof}

Let $X = U_0 \supseteq U_1 \supseteq \cdots$ be a basis of neighborhoods of $x$.  Define a sequence of cardinal numbers $(\lambda_n)_n$ by letting $\lambda_n = \card(U_n \setminus U_{n+1})$.  We now give another characterization of $\pi_1(\Sigma(X, x))$.

\begin{lemma}\label{characterizationastopprod}  $\pi_1(\Sigma(X, x)) \simeq \topprod_n F(\lambda_n)$.
\end{lemma}

\begin{proof}  Since $F(\lambda_n) \simeq F(U_n\setminus U_{n+1})$ we are really showing that $\pi_1(\Sigma(X, x)) \simeq \topprod_n F(U_n \setminus U_{n+1})$.  We use the characterization in Lemma \ref{unionofsubgroups} to define a homomorphism, and prove that this homomorphism is an isomorphism.

For each $K \in \K_X$ let $K_n = K \cap (U_n \setminus (U_{n+1}\setminus \{x\}))$ (each $K_n$ is finite).  The space $\Sigma(K, x)$ is either a finite wedge of circles or a homeomorph of the Hawaiian earring.  For each $x'\in X\setminus\{x\}$ let $\gamma_{x'}$ be a loop that traverses the circle $\Sigma(\{x, x'\}, x)$  exactly once by passing through the negative coordinates first and then the positives.  By the discussion in Section \ref{Prelim} we have an induced isomorphism $\phi_K: \pi_1(\Sigma(K, x)) \rightarrow \topprod_n F(K_n)$.    Moreover for $K, K'\in \K_X$ the map $\phi_{K\cup K'}: \pi_1(K\cup K', x) \rightarrow \topprod_n F(K_n \cup (K')_n)$ is an extension of $\phi_K$ and $\phi_{K'}$.  Thus there exists a homomorphism $\phi: \pi_1(\Sigma(X, x)) \rightarrow \topprod_n F(U_n \setminus U_{n+1})$, where the codomain we have written contains the codomain of each $\phi_K$.  As each $\phi_K$ is injective we know $\phi$ is injective.

For surjectivity of $\phi$, let $[W]\in \topprod_n F(U_n \setminus U_{n+1})$ be given.  Pick a representative $W \in [W]$.  For each $n \in \mathbb{N}$ the set $W^{-1} (F(U_n\setminus U_{n+1}))$ is finite.  Then for each $n\in \mathbb{N}$ there exists a finite set $K_n \subseteq U_n \setminus U_{n+1}$ such that each letter in $F(U_n\setminus U_{n+1})$ used in the word $W$ is actually in $F(K_n)$.  Letting $K = \bigcup_n K_n$ we see that our choice of notation $K_n$ matches the $K_n$ for $K$ as defined in the previous paragraph.  Moreover, $W$ is a representative of an element in $\topprod_n F(K_n)$, and  $\phi_K$ is onto $\topprod_n F(K_n)$, so $[W]$ is in the image of $\phi$.
\end{proof}

We state two results: the first is due to G. Higman \cite{Hi}, the second was proved by K. Eda \cite{E2}.

\begin{theorem}\label{Higman}  If $\phi : \topprod_n \mathbb{Z} \rightarrow F$ is a homomorphism with $F$ a free group, then for some $N \in \omega$ we have $\phi \circ p_N = \phi$ where $p_N$ is the retraction to $\ast_{n<N}\mathbb{Z}$.
\end{theorem}

\begin{theorem}\label{Chaselemma}  Let $(G_n)_n$ and $H_j$ ($j\in J$) be groups and $\phi :\topprod_n G_n \rightarrow \ast_{j\in J} H_j$ be a homomorphism to the free product of the groups $H_j$.  Then there exist $N\in \omega$ and $j\in J$ such that $\phi(\topprod_{n\geq N} G_n)$ is contained in a subgroup which is conjugate to $H_j$.
\end{theorem}

From Theorem \ref{Higman} we obtain the following:

\begin{lemma}\label{augmentedHigman}  If $\phi: \topprod_n F(\kappa_n) \rightarrow F$ is a homomorphism with $F$ a free group then for some $N$ we have $\phi\circ p_N = \phi$.
\end{lemma}

\begin{proof}  Suppose for contradiction that for $\phi: \topprod_n F(\kappa_n) \rightarrow F$ we have no $N\in \mathbb{N}$ such that $\phi\circ p_N = \phi$.  Then for each $N\in \mathbb{N}$ we have a word $W_N$ such that $[W_N] \in \topprod_{n>N} F(\kappa_n)$ and $\phi([W_N]) \neq 1$.  Then we can select finite sets $r_n \subseteq \kappa_n$ such that for all $N\in \mathbb{N}$ we have $[W_N]\in \topprod_n F(r_n) \leq \topprod_n F(\kappa_n)$.  Now $\phi\upharpoonright\topprod_n F(r_n)$ violates Theorem \ref{Higman} since $\topprod_n F(r_n) \simeq \topprod_n \mathbb{Z}$ as discussed in Section \ref{Prelim}.
\end{proof}

Let now $Y = V_0 \supseteq V_1 \supseteq \cdots$ be a basis of neighborhoods of $y$ and define the sequence $(\mu_n)_n$ of cardinal numbers by $\mu_n = \card(V_n \setminus V_{n+1})$.  For the forward direction of Theorem \ref{tpd-sbc} we assume that $\pi_1(\Sigma(X, x)) \simeq \pi_1(\Sigma(Y, y))$.  By Lemma \ref{characterizationastopprod} we have an isomorphism $\psi: \topprod_n F(\lambda_n) \simeq \topprod_n F(\mu_n)$.  If we can show that $(\lambda_n)_n \approx (\mu_n)_n$ we will be done, for a witness $f':\sqcup_{n = 0}^{\infty} \lambda_n \rightarrow \sqcup_{n=0}^{\infty}\mu_n$ induces a map $f:\sqcup_{n = 0}^{\infty} (U_n\setminus U_{n+1}) \rightarrow \sqcup_{n=0}^{\infty}(V_n\setminus V_{n+1})$ and by extending $f$ to have $f(x) = y$ we obtain a bijective map $f$ with $f$ continuous at $x$ and $f^{-1}$ continuous at $y$.

We check conditions (1)-(3) of Lemma \ref{equivalence}.  To avoid confusion we let $p_{N, X}$ and $p_X^N$ be the distinguished retractions on $\topprod_n F(\lambda_n)$ for each $N\in \mathbb{N}$ and similarly for $p_{N, Y}$ and $p_Y^N$ on $\topprod_n F(\mu_n)$.

To see that (1) holds we suppose without loss of generality that $(\mu_n)_n$ is eventually $0$.  Then $\topprod_{n} F(\mu_n) \simeq \ast_{n\leq M} F(\mu_n) \simeq F(\sum_{n \leq M} \mu_n)$.  Then the image of $\psi$ is a free group, so by Lemma \ref{augmentedHigman} we have $N$ such that $\psi \circ p_{N, X} = \psi$.  But $\psi$ is injective, so $p_{N, X}$ has trivial kernel, so that in fact $\lambda_n = 0$ for $n>N$.

To see that (2) holds suppose that $n \geq M$ implies $\mu_n <\kappa$ with $\kappa$ an infinite cardinal.  As $\psi:\topprod_n F(\lambda_n)\rightarrow (\ast_{n\leq M} F(\mu_n)) \ast (\topprod_{n>M} F(\mu_n))$, we know by Theorem \ref{Chaselemma} that there is some $N$ such that $\psi(\topprod_{n>N}F(\lambda_n))$ is a subgroup of either a conjugate of $\ast_{n\leq M} F(\mu_n)$ or a conjugate of $\topprod_{n>M} F(\mu_n)$.  By composing the isomorphism $\psi$ by conjugation in the group $\topprod_n F(\mu_n)$ we may assume that either $\psi(\topprod_{n>N}F(\lambda_n)) \leq \ast_{n\leq M} F(\mu_n)$ or  $\psi(\topprod_{n>N}F(\lambda_n)) \leq \topprod_{n>M} F(\mu_n)$.  In case $\psi(\topprod_{n>N}F(\lambda_n)) \leq \ast_{n\leq M} F(\mu_n)$, we know since $\psi$ is an isomorphism that in fact $\topprod_{n>N}F(\lambda_n)$ is a free group, so applying Lemma \ref{augmentedHigman} gives us an $M'>N$ such that $\lambda_n = 0<\kappa$ for all $n>M'$.  In case $\psi(\topprod_{n>N}F(\lambda_n)) \leq \topprod_{n>M} F(\mu_n)$ we have that the homomorphism $p_X^N \circ \psi^{-1}\upharpoonright\topprod_{n>M} F(\mu_n)$  is a surjection from $\topprod_{n>M} F(\mu_n)$  to $\topprod_{n>N}F(\lambda_n)$.  Then for each $N'>N$ the map $p_{N', X} \circ p_X^{N} \circ \psi^{-1}\upharpoonright\topprod_{n>M} F(\mu_n)$ is a surjection from $\topprod_{n>M} F(\mu_n)$ to the free group $\ast_{N' \geq n >N} F(\lambda_n)$, so by Lemma \ref{augmentedHigman} we have $M' >M$ such that $p_{N', X} \circ p_X^{N} \circ \psi^{-1}\circ p_{M'}:\topprod_{n>M} F(\mu_n) \rightarrow \ast_{N' \geq n >N} F(\lambda_n)$ satisfies  $p_{N', X} \circ p_X^{N} \circ \psi^{-1}\circ p_{M'}\upharpoonright\topprod_{n>M} F(\mu_n) = p_{N', X} \circ p_X^{N} \circ \psi^{-1}\upharpoonright\topprod_{n>M} F(\mu_n)$.  As $\im(p_{M'}\upharpoonright\topprod_{n>M} F(\mu_n)) = \ast_{M' \geq n >M}F(\mu_n)$ and $p_{N', X} \circ p_X^{N} \circ \psi^{-1}\upharpoonright\topprod_{n>M} F(\mu_n)$ is onto $\ast_{N' \geq n >N} F(\lambda_n)$, we know that the rank of $\ast_{M' \geq n >M}F(\mu_n)$ is at least that of $\ast_{N' \geq n >N} F(\lambda_n)$.  Thus in particular we have $\lambda_{N'}<\kappa$ for every $N'>N$ and we are done in this case as well.

To see that condition (3) holds we let $M\in\mathbb{N}$ be given.  For the homomorphism $p_{M, Y}\circ\psi: \topprod_n F(\lambda_n) \rightarrow \ast_{n \leq M}F(\mu_n)$ there exists $M_0' \in \mathbb{N}$ such that $p_{M, Y}\circ\psi\circ p_{M_0', X} = p_{M, Y}\circ\psi$.  Since $p_{M, Y}\circ\psi$ is a surjection and $\im(p_{M_0', X}) = \ast_{n\leq M_0'} F(\lambda_n)$, it must be that the rank of $\ast_{n\leq M_0'} F(\lambda_n)$ is at least that of $\ast_{n \leq M}F(\mu_n)$, so that $\sum_{n=0}^{M_0'}\lambda_n \geq \sum_{n=0}^M \mu_n$.  By performing a similar argument on $p_{M, X}\circ \psi^{-1}:\topprod_n F(\mu_n) \rightarrow \ast_{n \leq M}F(\lambda_n)$ we obtain an $M_1'$ with $\sum_{n=0}^{M_1'}\mu_n \geq \sum_{n=0}^M \lambda_n$, and $M' = \max\{M_0', M_1'\}$ gives us (3).

A referee has kindly pointed out a classification which follows from the proof of Theorem \ref{tpd-sbc}.  Firstly notice that if $(\lambda_n)_n$ and $(\mu_n)_n$ are sequences of cardinals such that neither $\max\{\lambda_n\}_{n\in \omega}$ nor $\max\{\mu_n\}_{n\in \omega}$ exists and $\sup\{\lambda_n\}_{n\in \omega} = \sup\{\mu_n\}_{n\in \omega}$ then $\topprod_n F(\mu_n) \simeq \topprod_n F(\lambda_n)$ (one can use the bijection defined in Lemma \ref{equivalence} Subcase 2 to determine the isomorphism).  In the setting of the previous sentence we let $\lambda = \sup\{\lambda_n\}_{n\in \omega}$ and define $G_{\lambda} = \topprod_n F(\lambda_n)$, so $G_{\lambda}$ does not depend on the choice of sequence $(\lambda_n)_n$ as was noted.

The classification is the following.

\begin{theorem}\label{classification}  Suppose that $X$ is first countable at $x$ and that $X$ is totally path disconnected Hausdorff.  Then $\pi_1(\Sigma(X, x))$ is isomorphic to at least one of the following.

\begin{enumerate}

\item $F(\kappa)$ for some cardinal $\kappa$

\item $\topprod_n F(\kappa)$ with $\kappa$ infinite

\item $G_{\lambda}$ with $\lambda$ infinite of countable cofinality

\item $F(\kappa) \ast G_{\lambda}$ with $\lambda$ infinite of countable cofinality and $\kappa \geq \lambda$

\item $F(\kappa) \ast \topprod_n F(\lambda)$ with $\lambda$ infinite and $\lambda < \kappa$

\end{enumerate}

\end{theorem}

\end{section}

\begin{section}{The proof of Theorem \ref{thmOmega}}\label{Omega}

We remind the reader of the statement of Theorem \ref{thmOmega}:

\begin{B}  For a pointed Hausdorff space $(X,x)$ which is first countable at the distinguished point the following properties are equivalent:
\begin{enumerate}
\item
$(X, x)$ is a horseshoe space
\item
The fundamental group $\pi_1(HA)$ of the harmonic archipelago embeds into $\pi_1(\Sigma(X,x))$
\item
The group of rationals $\mathbb Q$ embeds into $\pi_1(\Sigma(X,x))$
\item
$\pi_1(\Sigma(X,x))$ contains an infinitely divisible element

\item
$X$ is not of tpd-type, i.e.~there is no totally path disconnected space $Y$ such that $\pi_1(\Sigma(X,x)) \simeq \pi_1(\Sigma(Y,y))$
\end{enumerate}

\end{B}

We present the proof in a sequence of propositions.  The implication (2) $\Rightarrow$ (3) follows from the fact that $\mathbb{Q}$ embeds as a subgroup of $\pi_1(HA)$ (see \cite{CHM} for an explicit construction of such an embedding, or \cite{Ho}).  The claim (3) $\Rightarrow$ (4) is trivial.  We prove the implication (4) $\Rightarrow$ (5) in Proposition \ref{fourtofive}.  The claim (5) $\Rightarrow$ (1) will be the content of Proposition \ref{fivetoone}.  Finally, the implication (1) $\Rightarrow$ (2) will be shown in Proposition \ref{onetotwo}.

\begin{proposition}\label{fourtofive}  If $(Y, y)$ is totally path disconnected, Hausdorff and first countable at $y$ then $\pi_1(Y, y)$ has no infinitely divisible elements.  
\end{proposition}

\begin{proof}  We saw in Lemma \ref{characterizationastopprod} that $\pi_1(\Sigma(Y, y)) \simeq \topprod_{n} F(\mu_n)$ for some sequence $(\mu_n)_n$ of cardinals.  By the definition of the topologist product we know $g\in \topprod_{n} F(\mu_n)$ is not identity if and only if for some $N\in \mathbb{N}$ we have $p_N(g) \neq 1$.  If $g$ is infinitely divisible then so is $p_N(g)$, but $p_N(g)$ is in the free subgroup $\ast_{n\leq N}F(\mu_n)$ and free groups have no infinitely divisible elements.
\end{proof}

We build towards Proposition \ref{fivetoone}, first giving a construction of a graph and then proving a lemma.

Suppose $(X, x)$ is Hausdorff and also first countable at the basepoint.  Let $X = U_0 \supseteq U_1 \supseteq \cdots$ be a nesting basis of open neighborhoods of $x$, so in particular $\bigcap_n U_n = \{x\}$.  For $z\in X \setminus \{x\}$ let $g_0(z)$ be the maximum $n$ such that $z\in U_n$ and $g_1(z)$ be the minimum $n$ such that $z\notin \overline{U_n}$ (we are using the Hausdorff condition here).  Thus $g_0(z) <g_1(z)$.

\begin{construction}\label{graph}  Given a Peano continuum $Z \subseteq \Sigma(X, x)$ we construct a graph $\Gamma_Z$.  The construction of $\Gamma_Z$ will use certain arbitrary choices along the way.  Let $\check{C} = \check{X} \cap Z$ and $C = \chi(\check{C})$, so both $C$ and $\check{C}$ are compact but probably not Peano continua.  For $z\in C \setminus \{x\}$ we let $O_z' \subseteq (U_{g_0(z)} \setminus \overline{U_{g_1(z)}})\times (-\frac{1}{2}, \frac{1}{2}) \subseteq \mathring{X}$ be an open set such that $O_z' \cap Z$ is path connected and $\check{\iota}(z) \in O_z'$  (here we are using the fact that Peano continua are locally path connected).  Let $O_z = \chi(O_z' \cap \check{C})$ so that $O_z$ is a neighborhood of $z$ in the subspace $C \subseteq X$.  Notice that any two points $z_0, z_1 \in O_z$ can be connected by a path in $U_{g_0(z)} \setminus \overline{U_{g_1(z)}}$ since there is a path from $\check{\iota}(z_0)$ to $\check{\iota}(z_1)$ in $O_z' \cap Z$ which projects to a path in $U_{g_0(z)} \setminus \overline{U_{g_1(z)}}$ via $\chi$.

We define an open cover $\mathcal{O}$ of $C \setminus \{x\}$ by induction.  As $C \setminus U_1$ is compact and $\{O_z\}_{z\in C \setminus U_1}$ is an open cover, select a finite subcover $\mathcal{O}_0$.  The set $C \setminus (U_2 \cup \bigcup \mathcal{O}_0)$ is compact and covered by $\{O_z\}_{z\in C \setminus (U_2 \cup \bigcup \mathcal{O}_0)}$ so we pick a finite subcover $\mathcal{O}_1$.  Assuming $\mathcal{O}_n$ has been defined, notice that $C \setminus (U_{n+2} \cup \bigcup \bigcup_{j=0}^n \mathcal{O}_j)$ is compact and covered by $\{O_z\}_{z\in C \setminus (U_{n+1} \cup \bigcup \bigcup_{j=0}^n \mathcal{O}_j)}$ so we pick a finite subcover $\mathcal{O}_{n+1}$.  Let $\mathcal{O} = \bigcup_n \mathcal{O}_n$.  Define the graph $\Gamma_Z$ to have vertex set $\mathcal{O}$ and to have an edge between $O_{z_0}$ and $O_{z_1}$ if and only if $O_{z_0}' \cap O_{z_1}' \cap Z \neq \emptyset$.  Notice that each element of $\mathcal{O}$ is of finite valence since for each $O_z \in \mathcal{O}$ we have $O_z' \cap O_{z_0}' = \emptyset$ for each $O_{z_0} \in \mathcal{O} \setminus \bigcup_{n=0}^{g_1(z)} \mathcal{O}_n$.  
\end{construction}

\begin{remark}\label{niceremark}  If $J \subseteq \Gamma_Z$ is an infinite component then we can pick a tree $T \subseteq J$ such that $T$ has the same set of vertices as $J$ and all edges of $T$ are also edges in $J$.  Then $T$ would also be of finite valence, so by the K\"onig Infinity Lemma there exists in $T$, and therefore in $\Gamma_Z$, an infinite path $O_{z_0}, O_{z_1}, O_{z_2}, \ldots$ with $O_{z_m} \neq O_{z_n}$ if $m\neq n$.  Let $\rho_n: I \rightarrow U_{\min\{g_0(z_n), g_0(z_{n+1})\}}$ be a path from $z_n$ to $z_{n+1}$.  By pasting these paths together we obtain a path $\rho:I \rightarrow X$ from $z_0$ to $x$ by letting

\[
\rho(t) = \left\{
\begin{array}{ll}
\rho_n(\frac{t-(1-\frac{1}{n+1})}{\frac{1}{n+1} - \frac{1}{n+2}})
                                            & \text{if }t\in [1-\frac{1}{n+1}, 1-\frac{1}{n+2}], \\
x
                                            & \text{if }t = 1.
\end{array}
\right.
\]

\noindent That $\rho$ is indeed continuous follows from the pasting lemma and from the fact that for any $n\in \mathbb{N}$ we have $\rho(t)$ eventually in $U_n$.
\end{remark}

\begin{remark}\label{otherniceremark}  If $J \subseteq \Gamma_Z$ is a component then for any points $z, w \in (\bigcup J) \cap C$ there is a path in $X \setminus \{x\}$ from $z$ to $w$.  To see this, we let $O_{z_0}, O_{z_1}, \ldots O_{z_k}$ be a path in $J$ with $z \in O_{z_0}$ and $w \in O_{z_k}$.  We have a path in $X \setminus \{x\}$ from $z$ to $z_0$ obtained by projecting such a path in $O_{z_0}' \cap Z$ via $\chi$. There is also a path from $z_0$ to $z_1$, for $O_{z_0}' \cap O_{z_1}' \cap Z \neq \emptyset$ and so we can take a path from $z_0$ to a point in $O_{z_0}' \cap O_{z_1}' \cap Z$ and from that point to $z_1$ and project this path via $\chi$.  Similarly there are paths in $X \setminus \{x\}$ from $z_j$ to $z_{j+1}$ and from $z_k$ to $w$, so we gain a path from $z$ to $w$ by concatenation.
\end{remark}

For what follows, if $\mathcal{L}$ is a collection of disjoint sets a subset $Y \subseteq \bigcup \mathcal{L}$ is a \textbf{section} if $|Y \cap L| = 1$ for each $L \in \mathcal{L}$.

\begin{lemma}\label{sectioninclusion}  Suppose $(X, x)$ is Hausdorff and also first countable at $x$.  Let $Y \subseteq X$ be a section from $\pi_0(X)$ with $x\in Y$.  Then the homomorphism induced by inclusion $\iota_*: \pi_1(\Sigma(Y, x)) \rightarrow \pi_1(\Sigma(X, x))$ is injective.
\end{lemma}

\begin{proof}  Assume the hypotheses and let $X = U_0\supseteq U_1 \supseteq \cdots$ be a basis of open neighborhoods at $x$.  Since $(Y, x)$ is totally path disconnected, Hausdorff, and first countable at $x$ we know by Lemma \ref{unionofsubgroups} that $\pi_1(\Sigma(Y, x)) = \bigcup_{K\in \mathcal{K}_Y}\pi_1(\Sigma(K, x))$ in the notation of that lemma (so, $\mathcal{K}_Y$ is the collection of all compact subsets $K$ of $Y$ such that $K \setminus U_n$ is finite for each $n \in \mathbb{N}$).  Thus it suffices to show that $\iota_*\upharpoonright\pi_1(\Sigma(K, x))$ is injective for any $K\in \mathcal{K}_Y$, so fixing $K\in \mathcal{K}_Y$ we may simply assume $Y = K$.   We have by Lemma \ref{characterizationastopprod} an isomorphism of $\pi_1(\Sigma(K, x))$ with $\topprod_n F(\lambda_n)$, where $\lambda_n = \card(K \cap (U_n \setminus U_{n+1}))$.  By definition of $\mathcal{K}_Y$ we have each $\lambda_n$ finite.  Supposing for contradiction that $[W] \in \topprod_n F(\lambda_n)$ is a nontrivial element such that $[W] \in \ker(\iota_*)$ we select $N_0\in \mathbb{N}$ large enough that $p_{N_0}([W]) \neq 1$.  If $\gamma = \mathcal{R}([W])$ is a realization of $[W]$ (see Section \ref{Prelim}) then we have by assumption a nulhomotopy $H:I^2 \rightarrow \Sigma(X, x)$ of $\gamma$.  We may assume without loss of generality that $\Sigma(K, x) \subseteq H(I^2)$.

Perform Construction \ref{graph} of the vertex set $\mathcal{O}$ and graph $\Gamma_{Z}$ for the Peano continuum $Z = H(I^2)$, and as before we let $\check{C} = \check{X} \cap Z$ and $C = \chi(\check{C})$.  Letting $\{z_0, z_1, \ldots, z_q\} = K \setminus U_{N_0+1}$ we let $J_{z_j}$ be the path component in $\Gamma_{Z}$ of an $O \in \mathcal{O}$ such that $z_j \in O$.  For each $O \in \mathcal{O}$ we let $O'$ be the open set used in the construction such that $O' \cap Z$ is path connected and $\chi(O' \cap \check{C}) = O$.  For $0\leq j\leq q$ let $J_{z_j}' = \{O'\}_{O\in J_{z_j}}$.

For $z_{j_0} \neq z_{j_1}$ we have $(\bigcup J_{z_{j_0}}')\cap (\bigcup J_{z_{j_1}}') \cap Z = \emptyset$, since if there is some $O' \in J_{z_{j_0}}'$ and $V' \in J_{z_{j_1}}'$ with $O' \cap V' \cap Z \neq \emptyset$ then there is an edge between $O$ and $V$, so that in fact the components $J_{z_{j_0}}$ and $J_{z_{j_1}}$ are equal, and by Remark \ref{otherniceremark} we have a path in $X$ from $z_{j_0}$ to $z_{j_1}$ contrary to the fact that $z_{j_0}$ and $z_{j_1}$ are in different path components in $X$.  Also by Remark \ref{niceremark} each $J_{z_j}$ is finite since $x$ and $x_j$ are in distinct path components in $X$.  Thus we select $N \geq N_0$ large enough that $(U_{N} \times [-1, 1]) \cap \bigcup_{O' \in J_{z_{j}}'}O' = \emptyset$ for all $0 \leq j \leq q$.

Now $Z = H(I^2)$ includes into the set $$(U_{N} \times [-1, 1]) \cup \bigcup_{O \in \mathcal{O}}O' \cup \bigcup_{m \geq 2} X \times ([-1, -\frac{1}{m}) \cup (\frac{1}{m}, 1])$$ so by compactness we may select $M\in \mathbb{N}$ such that $H(I^2) \subseteq (U_{N} \times [-1, 1]) \cup \bigcup_{O \in \mathcal{O}} O' \cup X \times ([-1, -\frac{1}{M}) \cup (\frac{1}{M}, 1])$.

Define $f:H(I^2) \rightarrow  \Sigma(K \setminus (U_{N_0 + 1} \setminus \{x\}), x)$ by

\[
f((z, l)) = \left\{
\begin{array}{ll}
(z_j, Ml)
                                            & \text{if } (z, l)\in \bigcup J_{z_j}'\text{ and }l\in [-\frac{1}{M}, \frac{1}{M}], \\
x                                           & \text{otherwise}.
\end{array}
\right.
\]

We argue that $f$ is continuous.  Since the domain of $f$ is a Peano continuum, it is sufficient to prove that $f$ is sequentially continuous.  Thus we let $(y_r)_{r \in \mathbb{N}}$ be a sequence converging to $y$, with $y, y_r \in H(I^2)$.  We know that $f$ maps the open set $H(I^2) \cap ((U_{N} \times D_1) \cup X \times ([-1, -\frac{1}{M}) \cup (\frac{1}{M}, 1]))$ to the point $x$, so if $y$ is in this open set then the $f(y_r)$ are eventually equal to $x$.  Similarly the open set $H(I^2) \cap \bigcup_{O \in \mathcal{O} \setminus \bigcup_{j = 0}^q J_{z_j}} O'$ maps to $x$, so if $y$ is in this set we have $f(y_r)$ eventually equal to $x$.  Finally if $y \in H(I^2) \cap \bigcup_{j = 0}^q J_{z_j}'$ then for exactly one $0 \leq j \leq q$ we have $y \in H(I^2) \cap J_{z_j}'$ (we argued above regarding the disjointness of these intersections).  In this case the sequence $(y_r)_{r\in \mathbb{N}}$ is eventually in $H(I^2) \cap J_{z_j}'$ so we can assume without loss of generality that $(y_r)_{r} \subseteq H(I^2) \cap J_{z_j}'$.  The function $f$ maps $H(I^2) \cap J_{z_j}'$ to $\Sigma(\{z_j, x\}, x)$.  We may write in coordinates $y = (w, \eta)$ with $w \in X \setminus U_{N}$ and $\eta \in (-\frac{1}{2}, \frac{1}{2})$ and similarly $y_r = (w_r, \eta_r)$ with $w_r \in X \setminus U_{N}$ and $\eta_n \in (-\frac{1}{2}, \frac{1}{2})$, so that $w_r \rightarrow w$ and $\eta_r \rightarrow \eta$.  If $\eta \in [-1, -\frac{1}{M}] \cup [\frac{1}{M}, 1]$ then $f(y) = x$ and as either $f(y_r) = (z_j, M\eta_r)$ or $f(y_r) = x$ we see $f(y_r) \rightarrow x$ (in the first case, one notices that $M\eta_r$ gets arbitrarily close to $1$ or $-1$).  If $\eta \in (-\frac{1}{M}, \frac{1}{M})$ then without loss of generality $\eta_r \in (-\frac{1}{M}, \frac{1}{M})$ for all $r\in \mathbb{N}$ and $f(y_r) = (z_j, M\eta_r) \rightarrow (z_j, M\eta)$ and so we have shown that $f$ is continuous.

Note that $f\circ \iota:\Sigma(K, x) \rightarrow \Sigma(K \setminus (U_{N_0 + 1}\setminus \{x\}), x)$ is given by 

\[
(f\circ \iota)((z, l)) = \left\{
\begin{array}{ll}
(z_j, Ml)
                                            & \text{if } z = z_j\text{ and }l\in [-\frac{1}{M},\frac{1}{M}], \\
x                                           & \text{otherwise}.
\end{array}
\right.
\]

\noindent The map $T: \Sigma(K, x) \times I \rightarrow \Sigma(K \setminus (U_{N_0 + 1}\setminus \{x\}), x)$ defined by

\[
T((z,l), s) = \left\{
\begin{array}{ll}
(z, ((1 - s) + sM)l)
                                            & \text{if } z\notin U_{N+1}\text{ and }|((1 - s) + sM)l| \leq 1, \\
x                                           & \text{otherwise}.
\end{array}
\right.
\]

\noindent is a homotopy of the retraction $R: \Sigma(K, x) \rightarrow  \Sigma(K \setminus (U_{N_0 + 1}\setminus \{x\}), x)$ to $f\circ \iota$.  The induced retraction $R_*: \pi_1(\Sigma(K, x)) \rightarrow \pi_1(\Sigma(K \setminus (U_{N_0 + 1}\setminus \{x\}), x))$ corresponds to the retraction $p_{N_0}: \topprod_n F(\lambda_n) \rightarrow \topprod_{n \leq N_0} F(\lambda_n)$.  Thus on one hand we know that $R_*([\gamma]) = R_*([\mathcal{R}([W])]) \neq 1$ since $p_{N_0}([W]) \neq 1$, but on the other hand $R_*([\gamma]) = (f\circ \iota)_*([\gamma])$ is nulhomotopic in $\Sigma(K \setminus (U_{N_0 + 1}\setminus \{x\}), x)$ as witnessed by $f \circ H$.  Thus we obtain a contradiction.

%As an illustrative example, suppose that $X$ is Hausdorff, first countable at $x$ and has two path components, $\{x\}$ and $C$ where $x$ is in the closure $\overline{C}$.  Let $x_0 \in C$ and $\gamma: I \rightarrow \Sigma(X, x)$ be given by 
%
%
%
%\[
%\gamma(t) := \left\{
%\begin{array}{ll}
%(x_0, 2t - 1) 
%                                            & \text{if } t\in (0, 1), \\
%x                                           & \text{if } t\in \{0, 1\} .
%\end{array}
%\right.
%\]
%
%\noindent For a purported nulhomotopy $H$ of $\gamma$ our proof gives a continuous map $f: H(I^2) \rightarrow \Sigma(\{x, x_0\}, x)$ which takes $\gamma$ to a generator of $\pi_1(\Sigma(\{x, x_0\}, x))$, but $f \circ H$ would be a nulhomotopy of $f \circ \gamma$, which gives a contradiction.

\end{proof}

Suppose that $(X, x)$ is Hausdorff and first countable at $x$ and not a horseshoe space.  Then for each path component $C\in \pi_0(X)$ with $x\notin C$ there exists a minimal $n$ such that $U_{n+1}\cap C = \emptyset$.  Were this not the case we would get a sequence $(x_n)_n$ with $x_n \in C \cap U_n$.  We let sequence $(y_n)_n$ be defined by letting $y_0 = x_0$ and $y_{n+1} = x_n$, and since $(X, x)$ is not horseshoe we have for each $k\in \mathbb{N}$ an $n_k\in \mathbb{N}$ such that $n \geq n_k$ implies there exists a path in $U_k$ from $x_n$ to $y_n$ (else we could pick a subsequence and use $U = U_k$ to witness that $(X, x)$ is horseshoe).  This implies the existence of a path from $x_0$ to $x$, which contradicts $x\notin C$.

We say a section $Y \subseteq X$ of $\pi_0(X)$ with $x\in Y$ is \textbf{tight} if $z\in Y \setminus U_n$ implies $C \cap U_n  = \emptyset$ where $z\in C \in \pi_0(X)$.  We note that $(X, x)$ Hausdorff, first countable at $x$ and not a horseshoe space implies that $X$ has a tight section by selecting for each $C\in \pi_0(X)$ with $x\notin C$ the $n$ with $C \cap U_n \neq \emptyset$ and $C \cap U_{n+1} = \emptyset$ and picking $x_C\in C \cap U_n$.

We'll make use of the following very nice fact elucidated by Brazas and Gillespie (see \cite[Remark 3.5]{BrazGill}):

\begin{lemma}\label{Jeremynice}  If $(X, x)$ is a pointed Hausdorff space and $\gamma: I \rightarrow \Sigma(X, x)$ is a loop based at $x$, then $\gamma$ is homotopic to a reduced loop $\gamma_0: I \rightarrow \Sigma(X, x)$ at $x$ such that for each maximal interval $O \subseteq I \setminus \gamma_0^{-1}(\{x\})$ we have that $\chi \circ (\gamma_0 \upharpoonright O)$ is constant.

\end{lemma}

\begin{proposition}\label{fivetoone}
Let $(X,x)$ be Hausdorff and first countable at $x$ and not horseshoe and let $Y$ be a tight section.  Then

\begin{enumerate}[(a)]

\item the map $\iota_*: \pi_1(\Sigma(Y,x)) \rightarrow \pi_1(\Sigma(X,x))$ induced by inclusion is injective;

\item every loop $\gamma$ in $\Sigma(X,x)$ is homotopic to one in $\Sigma(Y,x)$, thus $\iota_*: \pi_1(\Sigma(Y,x)) \rightarrow \pi_1(\Sigma(X,x))$ is an isomorphism.
\end{enumerate}

\end{proposition}

\begin{proof}  We have already seen by Lemma \ref{sectioninclusion} that part (a) holds even if the section $Y$ is not tight.  For (b), it is immediate from the definition that since $X$ is not horseshoe there exists a nondecreasing function $j: \mathbb{N} \rightarrow \mathbb{N}$ such that $n \leq j(n)$ and if $z_0, z_1 \in U_{j(n)}$ are in the same path component in $X$ then they are in the same path component in $U_n$.  We take $\gamma: I \rightarrow \Sigma(X, x)$ to be a loop based at $x$.  By Lemma \ref{Jeremynice} we may assume that $\gamma$ is reduced and that for each maximal interval $O \subseteq I \setminus \gamma^{-1}(\{x\})$ we have $\chi \circ (\gamma \upharpoonright O)$ is constant, say constantly equal to $x_O \in X \setminus \{x\}$.  Let $\mathcal{O}$ denote the set of maximal intervals in $I \setminus \gamma^{-1}(\{x\})$.  For $O \in \mathcal{O}$ let $x_O$ be in path component $C \subseteq X$ and take $n \in \mathbb{N}$ to be the largest possible such that $x_O, x_C \in U_{j(n)}$, where $\{x_C\} = C \cap Y$.  Select a path $p_O: I \rightarrow X$ from $x_O$ to $x_C$ such that the image of $p_O$ lies entirely in $U_n$.

We define $\gamma_1: \bigcup \mathcal{O} \rightarrow X \setminus \{x\}$ and $\gamma_2: \bigcup \mathcal{O} \rightarrow D_1^\circ$ by the coordinates $\gamma(t) = (\gamma_1(t), \gamma_2(t))$ considered in $\mathring{X}$ (so $\gamma_1 = \chi \circ \gamma$ where defined).  Let $H: I^2 \rightarrow \Sigma(X, x)$ be given by 

\[
H(t, s) := \left\{
\begin{array}{ll}
x
                                            & \text{if } \gamma(t) = x, \\
(p_O(s), \gamma_2(t))                                           & \text{if } t\in O.
\end{array}
\right.
\]

\noindent The check that $H$ is continuous is straightforward, as for each $N \in \mathbb{N}$ there are only finitely many $O \in \mathcal{O}$ for which the image of $p_O$ has points outside of $U_N$.

\end{proof}

Next we provide a couple of lemmas towards Proposition \ref{onetotwo}.  We'll utilize the following well-known corollary of Ramsey's Theorem  (see \cite[Theorem 9.1]{J}).

\begin{theorem}\label{Ramsey}  Suppose that $J$ is a countably infinite set and $[J]^2 := \{\{j_0, j_1\}: j_0, j_1 \in J, j_0 \neq j_1\}$ is written as a disjoint union $[J]^2 = L_0 \sqcup L_1$.  Then there is an infinite subset $H \subseteq J$ with either $[H]^2 \subseteq L_0$ or $[H]^2 \subseteq L_1$.
\end{theorem}

\begin{lemma}\label{nicesequences}  If $(X, x)$ is a first countable horseshoe space there exists a neighborhood $U$ of $x$ and sequences $(x_n)_n$ and $(y_n)_n$ in $U$ converging to $x$ such that 

\begin{enumerate} \item There exists a path in $X$ from $x_n$ to $y_n$

\item There does not exist a path in $U$ from $x_n$ to 

\begin{enumerate}[(i)] \item $y_m$ for any $m\in \mathbb{N}$

\item $x_m$ for $m\neq n$

\item $x$

\end{enumerate}

\item There does not exist a path in $U$ from $y_n$ to

\begin{enumerate}[(i)]  \item $x_m$ for any $m\in \mathbb{N}$

\item $y_m$ for $m\neq n$

\item $x$

\end{enumerate}

\item Exactly one of the following holds:

\begin{enumerate}[(i)]
\item For $m\neq n$ there is no path in $X$ from $x_n$ to $x_m$

\item All elements in $\{x_n\}_{n\in \mathbb{N}} \cup \{y_n\}_{n\in \mathbb{N}}$ are in the same path component in $X$

\end{enumerate}
\end{enumerate}
\end{lemma}

\begin{proof}  By definition of horseshoe space there exists a neighborhood $V$ and sequences $(x_n')_n$ and $(y_n')_n$ in $V$ such that there exists a path from $x_n'$ to $y_n'$ in $X$ but no such path exists in $V$.  Let $U_0 = X \supseteq U_1 = V \supseteq U_2 \supseteq U_3 \supseteq \cdots$ be a nesting neighborhood basis at $x$ consisting of open sets.  We start with condition (4).  Define an equivalence relation $R$ on $\mathbb{N}$ by letting $(m, n) \in R$ if and only if there is a path in $X$ from $x_m'$ to $x_n'$.  By Theorem \ref{Ramsey} there exists an infinite subset $T \subseteq \mathbb{N}$ such that exactly one of the following holds:

\begin{enumerate}[(i)]  \item For all $m, n\in T$ we have $(m, n) \in R$.

\item For distinct $m, n\in T$ we have $(m, n)\notin R$.

\end{enumerate}

Thus by passing to the subsequence given by $T$ we may without loss of generality assume that (4) holds for the pair of sequences $(x_n')_n$ and $(y_n')_n$.  If (4) (i) holds then for at most one $n$ we have $x_n$ connected to $x$ via a path in $X$.  Then by passing to a further subsequence we get all $x_n$ in different path components from $x$ in the space $X$.  Furthermore, conditions (1) - (3) obviously follow for $U = V$ and $(x_n) = (x_n')_n$ and $(y_n)_n = (y_n')_n$.

Now we are left with the case where (4) (ii) holds for the sequences $(x_n')_n$ and $(y_n')_n$, so all elements in the set $S = \{x_n'\}_{n\in \mathbb{N}} \cup \{y_n'\}_{n\in \mathbb{N}}$ are in the same path component in $X$.  Define $h: S \times S \rightarrow \mathbb{N} \cup \{\infty\}$ by letting 

\begin{center}$h(z_0, z_1) = \sup\{n\in \mathbb{N}: \exists$ a path in $U_n$ from $z_0$ to $z_1\}$
\end{center}

\noindent  Were there to exist a sequence $(z_n)_n$ in $S$ such that $\liminf h(z_n, z_{n+1}) = \infty$ we would have paths $\rho_n$ in $U_{h(z_n, z_{n+1})}$ from $z_n$ to $z_{n+1}$ (here we define $U_{\infty} = \{x\}$).  By letting $\rho:I \rightarrow X$ be given by

\[
\rho(t) = \left\{
\begin{array}{ll}
\rho_n(n(n+1)(t-1+\frac{1}{n}))
                                            & \text{if }t\in [1-\frac{1}{n+1}, 1-\frac{1}{n+2}), \\
x
                                            & \text{if }t = 1.
\end{array}
\right.
\]

\noindent  we get a path $\rho$ from $z_0$ to $x$ within the open set $U_{\min \{h(z_n, z_{n+1})\}}$ which passes through all the points in $\{z_n\}_{n\in \mathbb{N}}$.  We treat further cases.

Suppose there is no subsequence $n_k$ in $\mathbb{N}$ for which $\liminf h(x_{n_k}', x_{n_{k+1}}') = \infty$.  We claim there exists an infinite set $T \subseteq \mathbb{N}$ and an $m\in \mathbb{N}$ such that $h(x_n', x_{n'}') \leq m$ for distinct $n, n'\in T$.  Supposing not, we define a symmetric relation $R_m = \{(n, n') \in \mathbb{N} \times \mathbb{N}: h(x_n', x_{n'}') \geq m\}$ and notice that $\mathbb{N} \times \mathbb{N}  = R_0 \supseteq R_1 \supset R_2 \supseteq \cdots$.  Theorem \ref{Ramsey} gives an infinite set $T_1 \subseteq \mathbb{N}$ such that one of the following occurs:

\begin{enumerate}[(a)]  \item $n, n'\in T_1 \Longrightarrow (n, n')\in R_1$

\item For distinct $n, n' \in T_1$ we have $(n, n')\notin R_1$
\end{enumerate}

\noindent and by our assumption we know it is (a) that holds.  As $T_1$ is infinite there exists by Theorem \ref{Ramsey} an infinite subset $T_2 \subseteq T_1$ such that $n, n'\in T_2 \Longrightarrow (n, n')\in R_2$ (again, by assumption).  Continuing in this way we get a sequence of infinite sets $\mathbb{N} \supseteq T_1\supseteq T_2\supseteq \cdots$  with $n, n'\in T_m$ implying $(n, n')\in R_m$.  But now we select $n_1\in T_1$, $n_2>n_1$ with $n_2\in T_2$, $n_3>n_2$ with $n_3 \in T_3$, etc. and obtain $\lim h(x_{n_k}', x_{n_{k+1}}') = \infty$, a contradiction.  Thus by passing to a subsequence and letting $U = U_m$ we may assume that $(x_n')_n$ is a sequence in $U$ such that there is no path in $U$ connecting $x_{n}'$ to $x_{n'}'$ for $n\neq n'$.  Moreover there is at most one $k$ for which $x_k'$ connects to $x$ by a path within $U$, so by passing to a further subsequence we may assume that none of the $x_n'$ connect to $x$ via a path in $U$.  We now let $x_n = x_{2n}'$ and $y_n = x_{2n+1}'$ and the conditions (1) - (4) are straightforward to check.  In case there is no subsequence $n_k$ in $\mathbb{N}$ for which $\liminf h(x_{n_k}', x_{n_{k+1}}') = \infty$ the same proof would yield $U$ and sequences $(x_n)_n$ and $(y_n)_n$ as desired.

Now suppose that for some subsequence $n_k$ in the naturals we have 

\begin{center}
$\liminf h(x_{n_k}',x_{n_{k+1}}') = \infty$.
\end{center}

\noindent Then by passing to a further subsequence we may assume that $h(x_n', x_{n+1}') \geq 1$ for all $n\in \mathbb{N}$ and that $\liminf  h(x_n', x_{n+1}') = \infty$.  Thus there exists a path in $U_1 = V$ from $x_0'$ to $x$ which passes through $x_1', x_2', x_3', \ldots$.  Now there cannot exist a subsequence $n_k$ in the naturals such that $\liminf h(y_{n_k}', y_{k_{k+1}}') = \infty$, for this would witness a path in $V$ which goes from some $y_{n_k}'$ to $x$ to $x_{n_k}'$, a contradiction.  Thus this puts us into an earlier case and we are done. 
\end{proof}

\begin{lemma}\label{bigcollapse}  For any pointed space $(Z, z)$ and $\epsilon\in (0,1]$  the map $f_{\epsilon}:  \Sigma(Z, z) \rightarrow \Sigma(Z, z)$ given by

\[
f_{\epsilon}((z_0, l))= \left\{
\begin{array}{ll}
(z_0, \frac{l}{\epsilon})
                                            & \text{if }\frac{l}{\epsilon} \in (-1, 1), \\
z
                                            & \text{otherwise}.
\end{array}
\right.
\]

\noindent is homotopic to the identity map.
\end{lemma}

\begin{proof}  The map $H: \Sigma(Z, z) \times [\epsilon, 1] \rightarrow \Sigma(Z, z)$ given by $H((z_0, l), t) = f_t(z_0, l)$ is a homotopy from $f_{\epsilon}$ to the identity map $f_1$.
\end{proof}

We note that $\pi_1(HA)$, the fundamental group of the harmonic archipelago, is isomorphic to the group $\topprod_n\mathbb{Z}/\langle\langle \{a_n\}_{n\in \mathbb{N}}\rangle\rangle$.  This follows by combining Theorem 5 and Proposition 13 of \cite{CHM}.

\begin{proposition}\label{onetotwo}  If $(X, x)$ is a first countable horseshoe space then the fundamental group of the harmonic archipelago embeds into $\pi_1(X, x)$.
\end{proposition}

\begin{proof}  Since $(X, x)$ is a horseshoe space we select an open neighborhood $U$ of $x$ and sequences $(x_n)_n$ and $(y_n)_n$ in $U$ satisfying the conclusion of Lemma \ref{nicesequences} .  Let $\gamma_{x_n}, \gamma_{y_n}: I \rightarrow \Sigma(X, x)$ be given by $\gamma_{x_n}(t) = (x_n, 2t-1)$ and similarly $\gamma_{y_n}(t) = (y_n, 2t-1)$, and $\gamma_n:I\rightarrow \Sigma(X, x)$ be given by $\gamma_{x_n}\ast \gamma_{y_n}^{-1}$.  Let $\Re: \topprod_n \mathbb{Z} \rightarrow \pi_1(\Sigma(X, x))$ be the realization map associated with the sequence $(\gamma_n)_n$ (see Section \ref{Prelim}).  We notice that each $\gamma_n$ is nulhomotopic in $\Sigma(X, x)$ since $\gamma_{x_n}$ is homotopic to $\gamma_{y_n}$ via the homotopy $H(t, s) = (\rho(s), 2t-1)$ where $\rho$ is a path from $x_n$ to $y_n$ in $X$.  Thus $a_n\in \ker(\Re)$ for each $n$ and $\langle\langle \{a_n\}_{n\in \mathbb{N}} \rangle\rangle \leq \ker(\Re)$.  We shall be done if we show $\langle\langle \{a_n\}_{n\in \mathbb{N}} \rangle\rangle \geq \ker(\Re)$.

Suppose that $[W]\in \ker(\Re)$ and fix a representative word $W \in [W]$ and order embedding $\iota:\overline{W} \rightarrow \mathcal{I}$.  By assumption there exists a nulhomotopy $H: I^2 \rightarrow \Sigma(X, x)$ of the loop $\gamma = \Re_{\iota}(W)$.  That is, for all $s, t \in I^2$ we have 

\begin{center}$H(t, 0) = \gamma(t)$

$H(0, s) = H(1, s) = H(t, 1) = x$

\end{center}

\begin{lemma}\label{gap}  There exists $N \in \mathbb{N}$ such that for any $n \geq N$, any path in $H(I^2)$ from $\check{\iota}(x_n)$ or $\check{\iota}(y_n)$ to $(X\setminus U) \times [-\frac{1}{2}, \frac{1}{2}]$ must travel through $X \times ([-1, -\frac{1}{2}) \cup (\frac{1}{2}, 1])$.
\end{lemma}

\begin{proof}  We know $H(I^2)$ is metrizable as the continuous Hausdorff image of a Peano continuum.  The sets $C_0 = ((X\setminus U)\times[-\frac{1}{2}, \frac{1}{2}]) \cap H(I^2)$ and $C_1= (\{x\}\cup \{\check{\iota}(x_n)\}_{n\in \mathbb{N}} \cup \{\check{\iota}(y_n)\}_{n\in \mathbb{N}})\cap H(I^2)$ are disjoint, closed and compact respectively in $H(I^2)$.  Thus there exists $d>0$ such that each element of $C_0$ is at least distance $d$ from any element of $C_1$.  Let $\{z_m\}_{m\in \mathbb{N}} \subseteq \{\check{\iota}(x_n)\}_{n\in \mathbb{N}} \cup \{\check{\iota}(y_n)\}_{n\in \mathbb{N}}$ be a set such that $m_0 \neq m_1 \Rightarrow z_{m_0}\neq z_{m_1}$ and such that for each $m\in\mathbb{N}$ there exists a path $\rho_m$ in $H(I^2)$ from $z_m$ fo $C_0$ which lies inside $(X \times [-\frac{1}{2}, \frac{1}{2}])\setminus \{x\}$.  Let $z_m'$ be the first point along $\rho_m$ which is distance $\frac{d}{2}$ away from $C_1$.  Then $z_{m_0}' \neq z_{m_1}'$ for $m_0 \neq m_1$ else we get a path in $U$ from $\chi(z_{m_0})$ to $\chi(z_{m_1})$ by passing along $\chi\circ \rho_{m_0}$ to $\chi(z_{m_0}')$ and then passing from $\chi(z_{m_0}')$ along $\chi\circ \rho_{m_1}$ back to $\chi(z_{m_1})$, which cannot exist by our choice of of sequences $(x_n)_n$ and $(y_n)_n$ as in Lemma \ref{nicesequences}.  Let $z\in H(I^2)$ be an accumulation point of the set $\{z_m'\}_{m\in \mathbb{N}}$.  Certainly the distance from $z$ to $C_1$ is $\frac{d}{2}$ and $z\in (U \times [-\frac{1}{2},  \frac{1}{2}]) \setminus \{x\}$.  By local path connectedness of $H(I^2)$ there exist distinct $m_0$ and $m_1$ and paths $\rho_0'$ and $\rho_1'$ in $(U\times (-\frac{3}{4}, \frac{3}{4}))\setminus \{x\}$ from $z$ to $z_{m_0}'$ and $z_{m_1}'$ respectively.  But now by composing these paths with $\chi$ we argue again that there is a path in $U$ from $\chi(z_{m_0})$ to $\chi(z_{m_1})$, a contradiction.
\end{proof}

Now using the $N$ specified by Lemma \ref{gap} we show that $p^N([W])$ is trivial in $\topprod_n\mathbb{Z}$.  The map $f_{\frac{1}{2}}\circ H$ is a nulhomotopy of $f_{\frac{1}{2}}\circ \gamma$.  Define $H'(t, s)$ to be $f_{\frac{1}{2}}\circ H(t, s)$ if for some $n >N$ there exists a path in $I^2 \setminus (f_{\frac{1}{2}}\circ H)^{-1}(x)$ from $(t, s)$ to an element of $I\times \{0\} \cap ((f_{\frac{1}{2}}\circ H)^{-1}(\{\check{\iota}(x_n)\}_{n>N} \cup \{\check{\iota}(y_n)\}_{n>N}))$ and let $H'(t, s) = x$ otherwise.  Note that $H'(t, 0) = f_{\frac{1}{2}} \circ \gamma'(t)$ where $\gamma' = \Re(p^N(W))$ and that $H'$ has range inside $\Sigma(U, x)$.  By the proof of Lemma \ref{bigcollapse} there is a homotopy of $\gamma'$ to $f_{\frac{1}{2}}\circ \gamma'$ which takes place entirely inside $\Sigma(U, x)$.  Thus $\gamma'$ is nulhomotopic in $\Sigma(U, x)$.

Let $\phi:\topprod_n \mathbb{Z} \rightarrow \topprod_n\mathbb{Z}$ be the endomorphism induced by the word map $g: a_n \mapsto a_{2n}a_{2n+1}^{-1}$.  This endomorphism is easily seen to be injective.  From the sequence of loops $(\delta_n)_n$ given by

\[
\delta_n = \left\{
\begin{array}{ll}
\gamma_{x_m}
                                            & \text{if }n = 2m, \\
\gamma_{y_m}
                                            & \text{if } n = 2m+1
\end{array}
\right.
\]

\noindent we get a realization map $\Re_{\delta}: \topprod_n\mathbb{Z} \rightarrow \pi_1(\Sigma(U, x))$.  By Lemma \ref{sectioninclusion} we have injectivity of the map $\Re_{\delta}$, and it is easily seen that $\Re_{\delta}\circ \phi(p^N( [W]))$ is homotopic in $\Sigma(U, x)$ to $\Re(p^N([W]))$.  As $\Re_{\delta} \circ \phi$ is injective we see that $p^N([W])$ is trivial in $\topprod_n\mathbb{Z}$.

\end{proof}

\end{section}

\section*{Acknowledgement}

The authors thank the referees for a careful reading of and suggestions for this paper, especially the incorrectness of the original proof of Proposition \ref{fivetoone} (b).  Also, deep gratitude to Jeremy Brazas is due for pointing out his paper with Gillespie \cite{BrazGill} which makes the proof of that proposition rather easy.

\bibliographystyle{amsplain}

\end{document}